\documentclass[a4paper]{article}

\usepackage{fullpage}
\usepackage{enumitem}
\usepackage{amsmath}
\usepackage{amssymb}
\usepackage{bbm}

\usepackage[francais,english]{babel}
\usepackage[latin1]{inputenc}
\usepackage{dsfont}
\usepackage{graphicx}
\usepackage{color}
\usepackage{epsfig}
\usepackage{amsthm,amssymb,amsbsy,amsmath,amsfonts,amssymb,amscd,mathrsfs}
\usepackage{colortbl}
\usepackage{verbatim}
\usepackage{bbm}
\usepackage[multiple]{footmisc}

\newtheorem{remark}{Remark}[section]
\newtheorem{prop}{Proposition}[section]

\newtheorem{defi}{Definition}[section]
\newtheorem{defs}{Definitions}[section]
\newtheorem{lemma}{Lemma}[section]

\newtheorem{theorem}{Theorem}[section]

\newtheorem{assumptions}{Assumptions}[section]

\def\ds{\displaystyle}
\def\Rset{\mathbb{R}}
\def\one{\mathbbm{1}}
\def\rk{\mbox{\em rk}\,}
\def\Re{\mbox{\em Re}\,}
\def\diag{\mbox{\em diag}\,}
\def\vand{\mbox{\em Vand}\,}

\title{Equivalence of finite dimensional input-output models of solute
  transport and diffusion in geosciences}
\author{{\sc A.~Rapaport}\footnote{INRA, UMR MISTEA,
    Montpellier, France.
E-mail: {\tt\small alain.rapaport@montpellier.inra.fr}},
{\sc A.~Rojas-Palma}\footnote{Departamento de Ingeniería Matemática \&
  Centro de Modelamiento Matemático (UMI 2807, CNRS), Universidad de
  Chile, Beauchef 851, Casilla 170-3, Santiago 3, Chile and UMR
  Inra-Supagro MISTEA, Montpellier, France.
E-mail: {\tt\small arojas@dim.uchile.cl}},
{\sc J.R.~de~Dreuzy}\footnote{CNRS, UMR G\'eosciences, Rennes, France.  E-mail: {\tt\small
    jean-raynald.de-dreuzy@univ-rennes1.fr}}
  and {\sc H.~Ram\'irez C.}\footnote{Departamento de Ingeniería Matemática \& Centro de Modelamiento Matemático (UMI 2807, CNRS), Universidad de Chile, Beauchef 851, Casilla 170-3, Santiago 3, Chile. E-mail: {\tt\small hramirez@dim.uchile.cl}}
}
\date{\today}

\begin{document}

\maketitle

\begin{abstract}
We show that for a large class of finite dimensional input-output
positive systems that represent networks of transport and diffusion of solute in
geological media, there exist equivalent {\em
  multi-rate mass transfer} and {\em multiple interacting continua} representations, which are quite
popular in geosciences. Moreover, we provide explicit methods to construct these
equivalent representations.
The proofs show that controllability property is playing a crucial
role for obtaining equivalence. These results contribute to our
fundamental understanding on the effect of fine-scale geological
structures on the transfer and dispersion of solute, and, eventually,
on their interaction with soil microbes and minerals. \\

\noindent {\bf Key-words.} Equivalent mass transfer models, positive linear
systems, controllability.\\

\noindent {\bf AMS subject classifications.} 93B17, 93B11, 15B48, 65F30.\\

\end{abstract}

\section{Introduction}
Underground media are characterized by their high surface to
volume ratio and by their slow solute movements that overall
promote strong water-rock interactions \cite{SDL05,VW76}. As a result, water quality strongly evolves with the degradation of anthropogenic contaminants and the dissolution of some minerals. Chemical reactivity is first determined by the residence time of solutes and the input/output behavior of the system, as most reactions are slow and kinetically controlled \cite{SM09,M11}. Especially important are exchanges between high-flow zones where solutes are transported over long distances with marginal reactivity and low-flow zones in which transport is limited by slow diffusion but reactivity is high because of large residence time \cite{HG95,CSBMGG98}. It is for example the case in fractured media where solute velocity can reach some meters per hour in highly transmissive fractures \cite{D86,F08} but 
 remains orders of magnitude slower in neighboring pores and smaller fractures giving rise to strong
dispersive effects \cite{G93,GWR92}.
More generally  wide variability of transfer times, high dispersion, and direct interactions
between slow diffusion in small pores and fast advection in much
larger pores are ubiquitous in soils and aquifers \cite{CS64}.
They are also the most characteristic features of underground
transport as long as it remains conservative (non-reactive). 
The dominance of these characteristic features up to some meters to
hundreds of meters have prompted the development of numerous
simplified models starting from the double-porosity concept \cite{WRA63}.

In double-porosity models, solutes move quickly by advection in a
first homogeneous porosity with a small volume representing focused
fast-flow channels and slowly by diffusion in a second large
homogeneous porosity. Exchanges between the two porosities is
diffusion-like, i.e. directly proportional to the differences in
concentrations. Such models have been widely extended to account not
only for one diffusive-like zone but for many of them with different
structures and connections to the advective zone \cite{HG95,PN85}.
Such extensions are thought to model both the widely varying transfer
times and the rich water-rock interactions. The two most famous ones are the
Multi-Rate Mass Transfer model (MRMT) \cite{CSBMGG98,HG95} and
Multiple INteracting Continua model (MINC).
They are made up of an infinity of diffusive zones deriving from
analytic solutions of the diffusion equation in layered, cylindrical
or spherical impervious inclusions (MRMT) or in series (MINC).
Between the single and infinite diffusive porosities of the
dual-porosity and these models, many intermediary models with finite
numbers of diffusive porosities have been effectively used and
calibrated on synthetic, field, or experimental data showing their
relevance and usefulness  \cite{DRBH13,BDC15,MMH01,WCSSD10,ZMHHPS04}.

Theoretical grounds are however missing to identify classes of
equivalent porosity structures, effective calibration capacity on
accessible tracer test data, and influence of structure on
conservative as well as chemically reactive transport. 
One can then naturally wonder which representation suits the best
experimental data, and if the two particular MRMT and MINC models are
not two restrictive structures. This is exactly the problem we
address in this work, from a theoretical approach based on linear algebra.

More precisely, we study the equivalence problem for 
a wide class of network structures and provide necessary and
sufficient conditions, making explicit the mathematical proofs.
We stick to the framework of stationary flows (in the mobile zone) and
assume water saturation in the immobile zones.
More concretely, we consider a system of $n$ compartments interconnected by
diffusion, whose water volumes $V_{i}$ ($i=1\cdots n)$ are assumed to
be constant over the time. One reservoir is subject to an advection of a
solute. We shall called {\em mobile zone} this particular reservoir,
and all the others $n-1$ reservoirs will be called {\em immobile
  zones} (see Fig. \ref{fig-network}). 
\begin{figure}[h]
\begin{center}
\includegraphics[width=8cm]{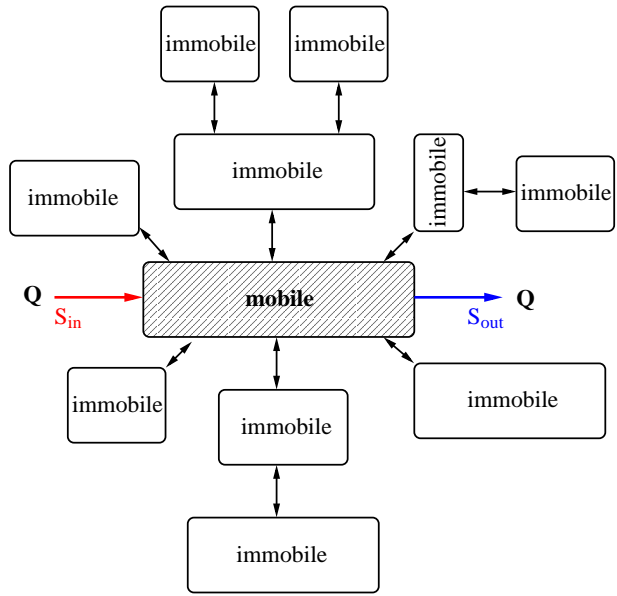}
\caption{\label{fig-network}Example of a network with one {\em mobile} zone}
\end{center}
\end{figure}

We aim at describing the time evolution of the concentrations $S_{i}$
($i=1\cdots n)$ of the solute in the $n$ tanks. The solute is injected
in tank $1$ with a water flow rate $Q$ at a concentration $S_{in}$, and
withdrawn from the same tank $1$ at the same water flow rate $Q$ with
a concentration $S_{out}=S_{1}$. Thus, the tank $1$ plays the role of
the {\em mobile zone}.
We represent this system by a system of $n$ ordinary equations:
\[
\begin{array}{rll}
\dot S_{1} & = & \ds \frac{Q}{V_{1}}(S_{in}-S_{1})+\sum_{j=2}^n
\frac{d_{1j}}{V_{1}}(S_{j}-S_{1})\\
\vdots & & \vdots\\
\dot S_{i} & = & \ds \sum_{j\neq i}\frac{d_{ij}}{V_{i}}(S_{j}-S_{i})\\
\vdots & & \vdots\\
\end{array}
\]
where the parameters $d_{ij}=d_{ji}$ ($i\neq j)$ denote the
diffusive exchange rates of solute between reservoirs $i$ and $j$.
For sake of simplicity, we shall assume
\[
\frac{Q}{V_{1}}=1
\]
which is always possible by a change of the time scale of the dynamics.
In the following we adopt an input-output setting in matrix form:
\begin{equation}
\label{input-output}
\begin{array}{lll}
\dot X & = & AX+Bu\\
y & = & CX
\end{array}
\end{equation}
where $X$ denotes the vector of the concentrations $S_{i}$ ($i=1\cdots
n$), $u$ the {\em input} that is $u=S_{in}$ and $y$ the {\em output}
$y=S_{out}=S_{1}$. The column and row matrices $B$ and $C$ are as follows
\[
B=\left[\begin{array}{c} 
1\\ 0\\ \vdots\\ 0
\end{array}\right] \quad \mbox{and} \quad
C=\left[\begin{array}{cccc} 
1 & 0 & \dots & 0
\end{array}\right] \ 
\]
and the matrix $A$ satisfies the following properties.\\

\begin{assumptions}
\label{H1}
There exist matrices $V$ and $M$ such that
\[
A=-BB^{t}-V^{-1}M
\]
where $V$ is a positive diagonal matrix and $M$ is a symmetric matrix that fulfills
\begin{enumerate}[label=\roman{*}.]
\item $M$ is irreducible (i.e. the graph with nodes $P_{i}$,
  and edges $\overrightarrow{P_{i}P_{j}}$ when $M_{ij}\neq 0$  is
  strongly connected)
\item $M_{ii}>0$ for any $i$
\item $M_{ij}\leq 0$ for any $i\neq j$
\item $\sum_{j}M_{ij}=0$ for any $i$
\end{enumerate}
\end{assumptions}

The diagonal terms of the matrix $V$ represent the volumes of the $n$
zones, and the off diagonal terms $M_{i,j}$ of the matrix $M$ are
the (opposite) of the diffusive exchange rate parameters between zones $i$
and $j$ (equal to $0$ if $i$ is not directly connected to $j$).
Properties i. and iv. are related to the connectivity of the graph
between zones and the mass conservation (i.e. Kirchoff's
law).
One can proceed to the following reconstruction of matrices $V$ and
$M$ from a given matrix $A$ that fulfills Assumptions \ref{H1}.

\begin{lemma}
\label{lemmaV}
Let $A$ fulfills Assumptions \ref{H1}. For any $i=2 \cdots n$,
there exists a permutation $\pi$ of $\{1,\cdots, n\}$ and an integer $l<n$ such that
$A_{\pi(j),\pi(j+1)}\neq 0$ $(j=1 \cdots l)$ with $\pi(1)=1$ and $\pi(l+1)=i$. 
Define then the numbers
\begin{equation}
\label{calculV}
V_{\pi(j+1)}=V_{\pi(j)}\frac{A_{\pi(j+1)\pi(j)}}{A_{\pi(j)\pi(j+1)}}
\quad j=1 \cdots l
\end{equation}
with $V_{1}=1$. Then $V$ is the diagonal matrix with $V_{1},\cdots,
V_{n}$ as diagonal entries, and $M=-V(A+BB^t)$.
\end{lemma} 

\begin{proof}
Under Assumptions \ref{H1}, there exists for any $i\neq 1$ a path from
node $1$ to $i$ in the direct graph associated to the matrix $A$,
visiting nodes only once, for which one can associate a permutation
$\pi$ such that $\pi(1)=1$, $\pi(l+1)=i$ with $A_{\pi(j)\pi(j+1)}\neq
0$ for $j=1\cdots l$, where $l<n$ is the path length.
Furthermore, one has (by Assumptions \ref{H1})
\[
\frac{A_{\pi(j)\pi(j+1)}}{V_{\pi(j)}}=\frac{A_{\pi(j+1)\pi(j)}}{V_{\pi(j+1)}}
    , \quad j=1,\cdots,l .
\]
One can then determine recursively the volumes
$V_{\pi(1)}\cdots V_{\pi(l+1)}$ 
with the expression (\ref{calculV})
from $V_{\pi(1)}=V_{1}$ that can be chosen equal to $1$.
From the diagonal matrix $V$ whose diagonal entries are the (non-null)
volumes $V_{1} \cdots V_{n}$, one can then reconstruct the matrix
$M=-V(A+BB^t)$.
\end{proof}

Matrices $A$ that fulfill Assumption \ref{H1} are compartmental
matrices, that have been extensively studied in
the literature (see for instance \cite{JS93,WC99}).
In the present work, we focus on properties for the
specific structure of compartmental matrices that we consider.
We first define in Section  \ref{section-configs} the two particular structures denoted
MRMT and MINC, and give some of their properties.
In Sections \ref{secMRMT} and
\ref{secMINC} we state and prove our main results about the equivalence of any
network structure with these two particular structures, under Assumption \ref{H1}
and an additional condition about the controllability.
Section \ref{section-minimal} discuss the crucial role played by
this controllability assumption to obtain the equivalence.
Finally, we draw conclusions with insights for geosciences.

\section{Notations and preliminary results}
\label{section-preliminary}

For sake of simplicity, we introduce the following notations
\begin{itemize}
\item for any vector $X \in \Rset^n$ and matrix $Z \in {\cal
    M}_{n,n}(\Rset)$, we denote
\[
\tilde X=[X_{i}]_{i=2\cdots n} \ , \qquad
\tilde Z=[Z_{ij}]\mathop{{}_{i=2\cdots n}}_{j=2\cdots n} \ .
\]
\item $\diag(X)$ denotes the diagonal matrix whose diagonal elements are
the entries of the vector $X$
\item we denote by $\vand(x_{1},\cdots,x_{m})$ the (square) Vandermonde matrix
\[
\vand(x_{1},\cdots,x_{m})=\left[\begin{array}{llll}
1 & x_{1} & \cdots & x_{1}^{m-1}\\
\vdots & \vdots & \vdots & \vdots\\
1 & x_{m} & \cdots & x_{m}^{m-1}
\end{array}\right]
\]

\item we define the vector in $\Rset^n$
\[
\one=\left[\begin{array}{c}1\\\vdots\\1
\end{array}\right] \ .
\]
\end{itemize}

\begin{lemma}
\label{lemmapositivity}
Under Assumptions \ref{H1}, the domain $\Rset_{+}^n$ is invariant by the
dynamics for any non-negative control $u$.
\end{lemma}

\begin{proof}
Take a vector $X$ that is on the boundary of $\Rset_{+}^n$ and set
$I=\{ i \in 1\cdots n \, \vert \, X_{i}=0\}$. At such a vector, one has
\[
\dot X_{i} = \sum_{j\notin} A_{i,j}X_{j}+B_{i}u \ , \quad i \in I
\]
Notice that the matrix $A$ is Metzler (that is all its non-diagonal
terms are non-negative) and $B$ is a non-negative vector. Consequently
one has
\[
\dot X_{i} \geq 0 \ ,  \quad i \in I
\]
which proves that any forward trajectory cannot leave the non-negative cone.
\end{proof}

\begin{remark}
Under Assumptions \ref{H1}, the linear system (\ref{input-output}) is
positive in the sense that for any non-negative initial state and
non negative control $u(\cdot)$, state and output are non-negative for
any positive time (see \cite{FR00}).
\end{remark}

\begin{lemma}
\label{lemmaMtilde}
Under Assumptions \ref{H1}, the matrix $\tilde M$ is symmetric definite positive.
\end{lemma}

\begin{proof}
The matrix $\tilde M$ is symmetric and consequently it is
diagonalizable with real eigenvalues. Its diagonal terms are positive
and off-diagonal negative or equal to zero. Furthermore one as
\[
\tilde M_{ii}=M_{i+1,i+1}=-\sum_{j\neq i+1}M_{i+1,j}=-\sum_{j\neq
  i}\tilde M_{i,j}-M_{i+1,1}\geq -\sum_{j\neq
  i}\tilde M_{i,j}
\]
The matrix $\tilde M$ is thus (weakly) diagonally dominant.
As each irreducible block of the matrix $\tilde M$ has to be connected to
the mobile zone (otherwise the matrix $A$ won't be irreducible), we
deduce that at least one line of each block has to be strictly
diagonally dominant. Then, each block is {\em irreducibly diagonally
dominant} and thus invertible by Taussky Theorem (see
\cite[6.2.27]{HJ85}).
Finally, the eigenvalues of the matrix $\tilde M$ belong to the Gershgorin discs
\[
G(\tilde M)=\bigcup_{i} \left\{ \lambda \in \Rset \, \vert \,
|\lambda-\tilde M_{i,i}|\leq \sum_{j\neq i}|\tilde M_{i,j}|\right\}
\]
and we deduce that each eigenvalues $\tilde \lambda_{i}$ of $\tilde
M$ are positive. The matrix $\tilde M$ is thus symmetric definite positive.
\end{proof}

\begin{lemma}
\label{lemmaAinversible}
Under Assumptions \ref{H1}, the matrix $A$ is non singular.
Furthermore, the dynamics admits the unique equilibrium $\one u$, for any constant control $u$
\end{lemma}

\begin{proof}
Let $X$ be a vector such that $AX=0$. Then, one has $BB^tX=-V^{-1}MX$
or equivalently
\[
MX=-V_{1}X_{1}B \ .
\]

Let us decompose the matrix $M$ as follows
\[
M=\left[\begin{array}{cc}
M_{11} & L\\
L' & \tilde M
\end{array}
\right]
\]
where $L$ is a row vector of length $n-1$.
Then equality $MX=-V_{1}X_{1}B$ amounts to write
\[
\left\{\begin{array}{l}
M_{11}X_{1}+L\tilde X=-V_{1}X_{1}\\
L'X_{1}+\tilde M \tilde X=0
\end{array}
\right.
\]
$\tilde M$ being invertible (Lemma \ref{lemmaMtilde}), one can write
$\tilde X=-\tilde M^{-1}L'X_{1}$ and thus $X_{1}$ has to fulfill
\[
(M_{11}-L\tilde M^{-1}L')X_{1}=-V_{1}X_{1}
\]
From Assumptions \ref{H1}, one has $M\one=0$ which gives
\[
\left\{\begin{array}{l}
M_{11}+L\tilde\one=0\\
L'+\tilde M\tilde\one=0
\end{array}
\right.
\]
that implies $M_{11}-L\tilde M^{-1}L'=0$. We conclude that one should
have $X_{1}=0$ and then $\tilde X=0$, that is $X=0$. The matrix $A$ is
thus invertible.

Finally, the system admits an unique equilibrium $X^\star=-A^{-1}Bu$
for any constant control $u$. As Assumptions \ref{H1} imply the equality
$A\one=-B$, we deduce that the equilibrium is given by $X^\star=\one u$. 

\end{proof} 

\begin{lemma}
\label{lemmaAtilde}
Under Assumptions \ref{H1}, the sub-matrix $\tilde A$ is diagonalizable
with real negative eigenvalues.
\end{lemma}

\begin{proof}
Notice first that the matrix $\tilde A$ can be written as 
$\tilde A=-\tilde V^{-1}\tilde M$. The matrix $\tilde V$
being diagonal with positive diagonal terms, one can consider its square root
$\tilde V^{1/2}$, defined as a diagonal matrix with 
  $\sqrt{\tilde V_{i}}$ terms on the diagonal, and its inverse $\tilde
  V^{-1/2}$. 
Then, one has
\[
\tilde V^{1/2}\tilde A\tilde V^{-1/2}=-
\tilde V^{-1/2}\tilde M\tilde V^{-1/2}
\]
which is symmetric. So $\tilde A$ is similar to a symmetric matrix, and thus
diagonalizable. Let $\lambda$ be an eigenvalue of $\tilde A$. There
exist an eigenvector $X\neq 0$ such that
\[
\tilde AX=\lambda X \Rightarrow X'\tilde V(\tilde AX)=\lambda
X'\tilde V X \Leftrightarrow X'\tilde MX=-\lambda
X'\tilde V X
\]
As $\tilde M$ is definite positive (Lemma \ref{lemmaMtilde}) as well as $\tilde
V$, we conclude that $\lambda$ has to be negative.
\end{proof}

\section{About controllability and observability}
\label{section-controllability}

We recall the usual definitions of controllability and observability
of single-input single-output systems $(A,B,C)$ of dimension $n$ (see
for instance \cite{K80}).

\begin{defs}
\begin{itemize}
\item[]
\item The {\bf controllability matrix} associated to the pair $(A,B)$ is
given by
\[
{\cal C}_{A,B}=[B,AB,\cdots,A^{n-1}B]
\]
\item The {\bf observability matrix} associated to the pair $(A,C)$ is
given by
\[
{\cal O}_{A,C}=\left[\begin{array}{c}
C\\
CA\\
\vdots\\
CA^{n-1}
\end{array}\right]
\]
\item A system $\dot X = AX+Bu$ 
is said to be {\bf controllable} when $\rk({\cal C}_{A,B})=n$, and {\bf
  observable} for $y=CX$ when $\rk({\cal O}_{A,C})=n$.
\item To a given triplet $(A,B,C)$, we associate the linear operator
  ${\cal F}_{A,B,C}: {\cal L}^2(\Rset_{+},\Rset) \mapsto {\cal
    L}^2(\Rset_{+},\Rset)$ that is defined as $y(\cdot)={\cal
    F}_{A,B,C}[u(\cdot)]$ with $y(\cdot)=CX(\cdot)$ where $X(\cdot)$
  is solution of $\dot X=AX+Bu(\cdot)$ for the initial condition
  $X(0)=0$. We say that a triplet $(A,B,C)$ is a {\bf minimal
  representation} if among all the triplets $(A^\dag,B^\dag,C^\dag)$ such that
  ${\cal F}_{A^\dag,B^\dag,C^\dag}={\cal F}_{A,B,C}$, the dimension of
  $A$ is minimal.
\end{itemize}
\end{defs}
We recall a well known result of the literature on linear input-output
systems \cite[2.4.6]{K80}.
\begin{theorem}(Kalman)
A representation $(A,B,C)$ is minimal if and only if the pairs
$(A,B)$ and $(A,C)$ are respectively controllable and observable.
\end{theorem}

\bigskip

The particular structures of the matrices $A$, $B$ and $C$ that we
consider allow to show the following property.

\begin{lemma}
\label{lemmacontrollable}
Under Assumptions \ref{H1}, $(A,B)$ controllable is equivalent
to $(A,C)$ observable.
\end{lemma}

\begin{proof}
Notice first that one has
\[
VAV^{-1}=-BB^t-MV^{-1}=(-BB^t-V^{-1}M)^\prime=A'
\]
and by recursion
\[
VA^kV^{-1}=(A')^k
\]
Then, one can write
\[
{\cal O}'_{A,C} = [B,A'B,(A')^{2}B,\cdots] = V[V^{-1}B,AV^{-1}B,A^{2}V^{-1}B,\cdots]
\]
But one has
\[
V^{-1}B=V_{1}^{-1}B
\]
Thus
\[
V_{1}{\cal O}_{A,C}^\prime=V[B,AB,A^{2}B,\cdots]=V{\cal C}_{A,B}
\]
and we conclude
\[
\rk({\cal O}_{A,C}) = \rk({\cal C}_{A,B}).
\]
\end{proof}

We also recall a nice result about tridiagonalization of single-input
single-output systems, from \cite[Lemma 2.2]{GKV92}.

\begin{prop}
\label{PropGolub}
Let $T$ be an invertible transformation, then one has
\[
T^{-1}AT=\left[\begin{array}{ccccc}
\star & y_{2} & & &  \makebox(-10,-10){\text{\em \Large 0}} \\
x_{2} & \ddots & \ddots & & \\
& \ddots & \ddots & \ddots & \\
& & \ddots & \ddots & y_{n}\\
\makebox(20,20){\text{\em \Large 0}} &  & & x_{n} & \star
\end{array}\right], \quad 
T^{-1}B=\left[\begin{array}{c}x_{1}\\0\\\vdots\\\vdots\\0
\end{array}\right], \quad
CT=\left[\begin{array}{ccccc} y_{1} & 0 & \cdots & \cdots & 0
\end{array}\right]
\]
with $x_{i}\neq 0$ and $y_{i}\neq 0$ $(i=1\cdot n)$ if and only if
\[
T^{-1}{\cal C}_{A,B}=\left[\begin{array}{cccc}
c_{1} & \star & \cdots & \star\\
 & \ddots & \ddots & \vdots\\
& & \ddots & \star\\
\makebox(20,20){\text{\em \Large 0}} & & & c_{n}
\end{array}\right] \mbox{ and }
{\cal O}_{A,C}T=\left[\begin{array}{cccc}
o_{1} & & & \makebox(-10,-10){\text{\em \Large 0}}\\
\star & \ddots & & \\
\vdots & \ddots & \ddots & \\
\star & \cdots & \star & o_{n}
\end{array}\right]
\]
with $c_{i}\neq 0$, $o_{i}\neq 0$  $(i=1\cdot n)$ . 
Furthermore, one has
$x_{1}=c_{1}$, $y_{1}=o_{1}$, $x_{i+1}=c_{i+1}/c_{i}$,
$y_{i+1}=o_{i+1}/o_{i}$ $(i=1\cdots n-1)$.
\end{prop}

\section{The Multi-Rate Mass Transfer and Multiple INteracting
  Continua configurations}
\label{section-configs}

We consider two particular structures of networks whose $(A,B,C)$
representations fulfill Assumption \ref{H1}.

\begin{defi}
A matrix $A$ that
fulfills Assumptions \ref{H1} and such that the sub-matrix
\[
\tilde
A=[A_{ij}]\mathop{{}_{i=2\cdots n}}_{j=2\cdots n}
\]
is diagonal is called a MRMT
(multi-rate mass transfer) matrix.
\end{defi}

MRMT matrices correspond to particular {\em arrow} structure of the
matrix $A$:
\[
A=\left[\begin{array}{cccccc}
-\frac{Q}{V_{1}}-\sum_{i}\frac{d_{1i}}{V_{1}} & \frac{d_{12}}{V_{1}} & \cdots &
\cdots & \cdots & \frac{d_{1n}}{V_{1}} \\[3mm]
\frac{d_{12}}{V_{2}} & -\frac{d_{12}}{V_{2}} & 0 & \cdots & \cdots & 0\\[3mm]
\vdots & 0 & \ddots & \ddots &  & \vdots \\[3mm]
\vdots & \vdots & & \ddots & \ddots & 0 \\[3mm]
\frac{d_{1n}}{V_{n}} & 0 & \cdots & \cdots & 0 & -\frac{d_{1n}}{V_{n}}
\end{array}\right]
\]
or {\em star} connections f the immobile part of depth one,
where all the immobile zones are connected to the mobile one (see
Fig. \ref{fig-MRMT}).
\begin{figure}[h]
\begin{center}
\includegraphics[width=8cm]{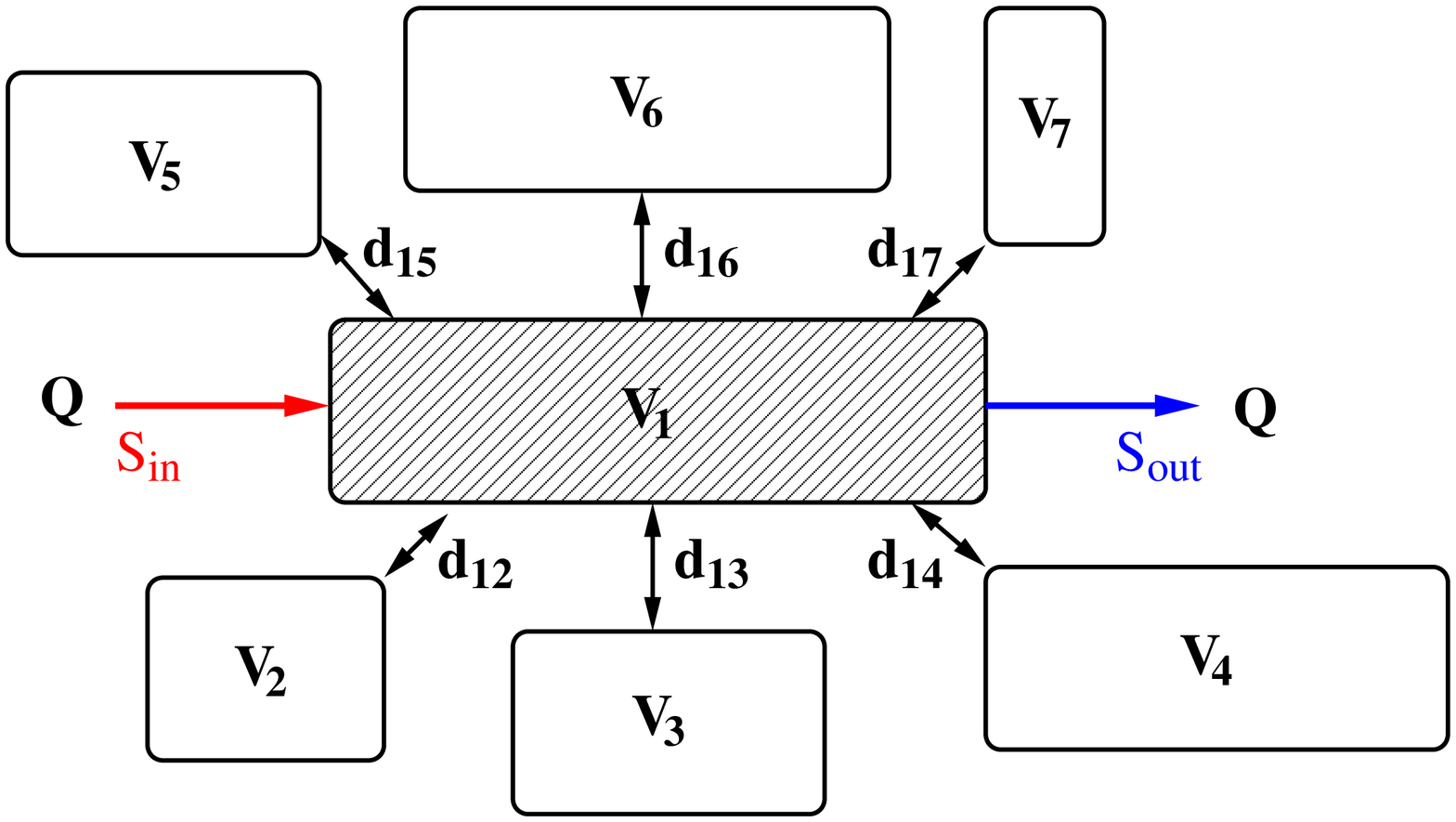}
\caption{\label{fig-MRMT}Example of a MRMT network}
\end{center}
\end{figure}

\bigskip

\begin{defi}
A matrix $A$ that
fulfills Assumptions \ref{H1} and which is tridiagonal is called a MINC
(Multiple INteracting Continua) matrix.
\end{defi}

MINC matrices correspond to particular structure:
\[
A=\left[\begin{array}{rcccccc}
\frac{Q}{V_{1}}-\frac{d_{12}}{V_{1}} & \frac{d_{12}}{V_{1}} &  &
 &  & \\[3mm]
\frac{d_{12}}{V_{2}} & -\frac{d_{12}+d_{23}}{V_{2}} &
\frac{d_{23}}{V_{2}} &  &   & \makebox(0,0){\text{\huge 0}}  \\[3mm]
 & \ddots & \ddots & \ddots & & & \\[3mm]
 \makebox(0,0){\text{\huge0}} & & \ddots & \ddots & \ddots & &  \\[3mm]
  &  &  &  & \frac{d_{n-1,n}}{V_{n}} & -\frac{d_{n-1,n}}{V_{n}}
 \end{array}\right]
\]
where the immobile parts are connected {\em in series}, of length
$n-1$, one of them being connected to the mobile zone (see Fig. \ref{fig-MINC}).
\begin{figure}[h]
\begin{center}
\includegraphics[width=8cm]{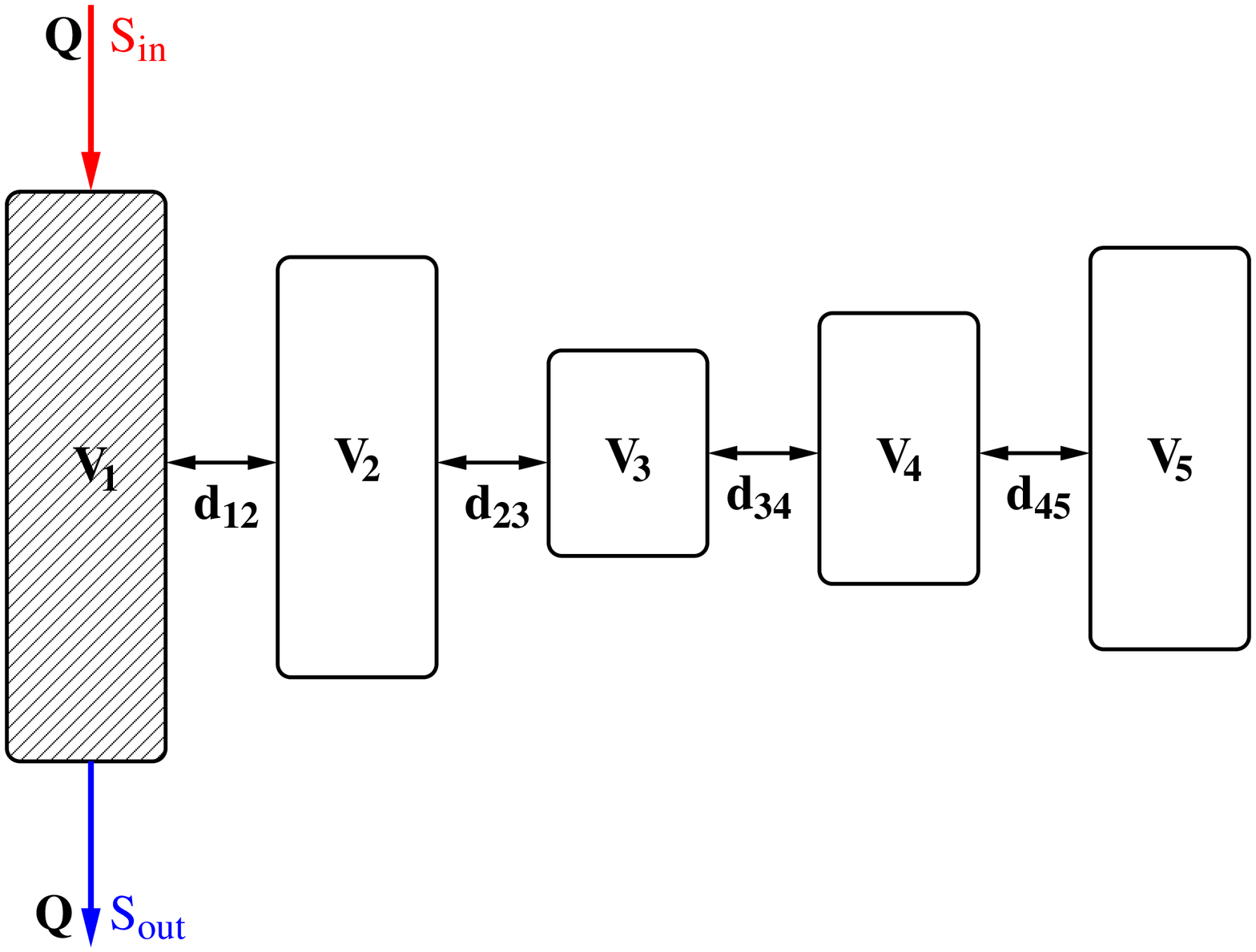}
\caption{\label{fig-MINC}Example of a MINC network}
\end{center}
\end{figure}

In the following, we give properties on eigenvalues for MRMT matrices
only, because it is easier to be proved for this particular structure. In
the next section, we shall show that MINC and MRMT structures are
indeed equivalent, and as a consequence eigenvalues of MINC matrices
fulfill the same properties.

\begin{lemma}
A MRMT matrix is Hurwitz (i.e. all the real parts of its eigenvalues
are negative).
\end{lemma}

\begin{proof}
Take a number 
\[
\gamma>\max\left(\frac{Q}{V_{1}}+\sum_{i}\frac{d_{1i}}{V_{1}},\frac{d_{12}}{V_{2}},\cdots,
\frac{d_{1n}}{V_{n}}\right) \ .
\]
Then the matrix $\gamma I +A$ is an irreducible non-negative
matrix. From Perron-Frobenius Theorem (see \cite[Th 1.4]{BP94}), $r=\rho(\gamma I +A)$ is a
single eigenvalue of $\gamma I +A$ and there exists a positive
eigenvector associated to this eigenvalue. 
That amounts to claim that there exists a positive
eigenvector $X$ of the matrix $A$ for a single (real) eigenvalue
$\lambda=r-\gamma$, and furthermore that any other eigenvalue $\mu$ of
$A$ is such that
\[
-r<\gamma+\Re(\mu)<r \; \Rightarrow \; \Re(\mu)<\lambda \ .
\]
From the particular structure of MRMT matrix, such a vector $X$
has to fulfill the equalities
\[
\begin{array}{cll}
\ds -\frac{Q}{V_{1}}+\sum_{i=2}^n
\frac{d_{1i}}{V_{1}}(X_{i}-X_{1}) & = & \lambda X_{1}\\[4mm]
d_{1i}(X_{1}-X_{i})& = & \lambda V_{i} X_{i} \qquad (i=2\cdots n)
\end{array}
\]
from which one obtains
\[
 -\frac{Q}{V_{1}} =
 \lambda\left(X_{1}+\sum_{i=2}^n\frac{V_{i}}{V_{1}}X_{i}\right) \ .
\]
The vector $X$ being positive, we deduce that $\lambda$ is negative.
\end{proof}

As $\one u$ is an equilibrium of the system (\ref{input-output}) for any constant
control $u$, this Lemma allows then to claim the following result.

\begin{lemma}
For any constant control $u$, $\one u$ is a globally
exponentially stable of the dynamics (\ref{input-output}).
\end{lemma}

Finally, we characterize the minimal MRMT representations as follows.

\begin{lemma}
For a minimal representation $(A,B,C)$ where $A$ is MRMT,
the eigenvalues of the matrix $\tilde A$ are distinct.
\end{lemma}

\begin{proof}
The eigenvalues of $\tilde A$ for the MRMT structure are 
$\lambda_{i}=-d_{1i}/V_{i}$ ($i=2\cdots n$).
If there exist $i\neq j$ in $\{2,\cdots,n\}$ such that
$\lambda_{i}=\lambda_{j}=\lambda$, one can consider the variable
\[
S_{ij}=\frac{V_{i}}{V_{i}+V_{j}}S_{i}+\frac{V_{j}}{V_{i}+V_{j}}S_{j}
\]
instead of $S_{i}, S_{j}$ and write equivalently the dynamics in
dimension $n-1$:
\[
\begin{array}{rll}
\dot S_{1} & = & \ds \frac{Q}{V_{1}}(S_{in}-S_{1})+\sum_{k\geq
  2,k\neq i,j}
\frac{d_{1k}}{V_{1}}(S_{k}-S_{1}) + \frac{d_{1ij}}{V_{ij}}(S_{ij}-S_{1})\\
\vdots & & \vdots\\
\dot S_{k} & = & \ds \frac{d_{1k}}{V_{k}}(S_{1}-S_{k})
\qquad k \in \{2,\cdots,n\}\setminus\{i,j\}\\
\vdots & & \vdots \\
\dot S_{ij} & = & \ds \frac{d_{1ij}}{V_{ij}}(S_{1}-S_{ij})\\
\end{array}
\]
with $V_{ij}=V_{i}+V_{j}$ and $d_{1ij}=-(V_{i}+V_{j})\lambda$, which
show that $(A,B,C)$ is not minimal.
\end{proof}

In the coming sections, we address the equivalence problem of any
network structure that fulfill Assumption \ref{H1} with either a MRMT
or  a MINC structure. There are many known ways to diagonalize the sub-matrix
$\tilde A$ or tridiagonalize the whole matrix $A$ to obtain matrices
similar to $A$ with an arrow or tridiagonal structure. The
remarkable feature we prove is that there exist such 
transformations that preserve the signs of the entries of the
matrices (i.e. Assumption \ref{H1} is also fulfilled in the new coordinates)
so that the equivalent networks have a physical interpretation.\\

For a system $(A,B,C)$ that fulfills Assumption \ref{H1}, we adress in
the two coming Sections the problem of finding equivalent MRMT or MINC
structure (which are two other positive realizations \cite{BF04,NM05,B13} of the
input-output system).

\section{Equivalence with MRMT structure}
\label{secMRMT}

We first give sufficient conditions to obtain the equivalence with MRMT.

\begin{prop}
\label{mainprop}
Under Assumption \ref{H1}, take an invertible matrix $P$ such that
$P^{-1}\tilde A P=\Delta$, where $\Delta$ is diagonal.
If all the entries of the vector $P^{-1}\tilde\one$ are non-null and
the eigenvalues of $\tilde A$ are distinct, the
matrix
\[
R=\left[\begin{array}{cc}
1 & 0\\
0 & -P\Delta^{-1}\diag(P^{-1}A(2:n,1))
\end{array}\right]
\]
is invertible and such that $R^{-1}\!AR$ is a MRMT matrix.
\end{prop}

\begin{proof}
Take a general matrix $A$ that fulfills Assumption \ref{H1}.
From Lemma \ref{lemmaAtilde}, $\tilde A$ is diagonalizable with $P$ such
that $P^{-1}\tilde A P=\Delta$ where $\Delta$ is a diagonal matrix.
Let $G$ be the diagonal matrix
\[
G=-\Delta^{-1}\diag(P^{-1}A(2:n,1))
\]
and define $\tilde R=PG$.\\
Notice that one has $A\one=-B$ from Assumptions \ref{H1}.
The $n-1$ lines of this equality
gives $A(2:n,1)+\tilde A\tilde\one=0$ and one can write
\[
\begin{array}{ll}
 & P^{-1}A(2:n,1)+P^{-1}\tilde A\tilde\one=0\\
\Leftrightarrow & P^{-1}A(2:n,1)+P^{-1}\tilde A P
P^{-1}\tilde\one=0\\
\Leftrightarrow & P^{-1}A(2:n,1)+\Delta P^{-1}\tilde\one=0
\end{array}
\]
Thus having  all the entries of the vector $P^{-1}A(2:n,1)$ non-null
is equivalent to have  all the entries of the vector
$P^{-1}\tilde\one$ non null.\\

All the entries of the vector $P^{-1}A(2:n,1)$ being non-null,
$\tilde R$ is invertible and one has
\[
\tilde R^{-1}\tilde A\tilde R=G^{-1}P^{-1}\tilde A PG=G^{-1}\Delta G = \Delta.
\]
One can then consider the matrix $R \in {\cal M}_{n,n}$ defined as
\[
R=\left[\begin{array}{cc}
1 & 0\\
0 & \tilde R
\end{array}\right]
\quad \mbox{with} \quad
R^{-1}=\left[\begin{array}{cc}
1 & 0\\
0 & \tilde R^{-1}
\end{array}\right]
\]
One has
\[
R^{-1}AR=\left[\begin{array}{cc}
A_{11} & A(1,2:n)\tilde R\\[4mm]
\tilde R^{-1}A(2:n,1) & \Delta
\end{array}\right]
\]
We show now that the matrix $R^{-1}\!AR$ fulfills Assumptions \ref{H1}.\\

One has straightforwardly
\[
R^{-1}AR=-B^tB-R^{-1}V^{-1}MR \ .
\]
As the irreducibility of the matrix $V^{-1}M$ is preserved by the change
of coordinates given by $X\mapsto R^{-1}X$, Property i. is fulfilled.\\

The diagonal terms of $-(R^{-1}AR+BB^t)$ are $-A_{11}-1$ (which is
positive) and the diagonal of $-\Delta$ which is also positive. Property
ii. is thus satisfied.\\ 

We have now to prove that column $\tilde R^{-1}A(2:n,1)$ and row
$A(1,2:n)\tilde R$ are positive to show Property iii.
From the definition of the matrix $G$, one has
\[
\Delta=-G^{-1}\diag(P^{-1}A(2:n,1))=-\diag(\tilde R^{-1}A(2:n,1))
\]
and thus one has
\[
\tilde R^{-1}A(2:n,1)=-\Delta\tilde\one
\]
As the diagonal terms of $\Delta$ are negative, we deduce that the
vector $\tilde R^{-1}A(2:n,1)$ is positive.
As the matrix $V\!A$ is symmetric, one can write
$V_{11}A(1,2:n)=A(2:n,1)^{\prime}\tilde V$
and then
\[
\left(V_{11}A(1,2:n)\tilde R\right)^{\prime}=\tilde
  R^{\prime}\tilde VA(2:n,1)=-\tilde R^{\prime}\tilde V\tilde R\Delta\tilde\one
\]
Notice that the matrix $\tilde R^{\prime}\tilde V\tilde R$ can be
written $T'T$ with $T=\tilde V^{1/2}\tilde R$, and that the matrix $T$
diagonalizes the matrix $S=\tilde V^{1/2}\tilde A \tilde V^{-1/2}$:
\[
T^{-1}ST=\tilde R^{-1}\tilde A\tilde R=\Delta
\]
The matrix $S$ being symmetric, it is also diagonalizable with a
unitary matrix $U$ such that $U'SU=\Delta$.
As the eigenvalues of $\tilde A$ are distinct, their eigenspaces are
one-dimensional and consequently the columns of any
matrix that diagonalizes $S$ into $\Delta$ have to be
proportional to corresponding eigenvectors.
So the matrix $T$ is of the form $UD$ where $D$ is a non-singular
diagonal matrix. This implies that the matrix 
$\tilde R^{\prime}\tilde V\tilde R$ is equal to $D^2$, which is a
positive diagonal matrix. As $-\Delta\tilde\one$ is a positive vector, we deduce that the
entries of $A(1,2:n)\tilde R$ are positive.\\

Notice that $\tilde\one$ is necessarily an eigenvector of $\tilde R^{-1}$
(or $\tilde R$) for the eigenvalue $1$:
as one has $A\one=-B$, one has also $A(2:n,1)=-\tilde A\tilde\one$ and then
\[
\tilde R^{-1}\tilde\one=-\tilde R^{-1}\tilde A^{-1}A(2:n,1)
=-\Delta^{-1}\tilde R^{-1}A(2:n,1)
=-\Delta^{-1}\diag(\tilde R^{-1}A(2:n,1))\tilde\one
=\tilde\one
\]
Finally, one has
\[
(R^{-1}AR+BB^t)\one=R^{-1}A\one+B=-R^{-1}B+B=0
\]
which proves that Property iv. is verified.

\end{proof}

We come back to the condition required by Proposition
\ref{mainprop} and show that it is necessarily fulfilled for
minimal representations (we recall from Lemma \ref{lemmacontrollable}
that controllability implies a minimal representation in our framework).

\begin{prop}
\label{propminimal}
Under Assumptions \ref{H1}, the entries of the vector
$P^{-1}\tilde \one$ are non null for any $P$ such that
$P^{-1}\tilde AP=\Delta$ with $\Delta$ diagonal, 
when the pair $(A,B)$ is controllable.
Furthermore, the eigenvalues of $\tilde A$ are distinct.
\end{prop}

\begin{proof}
From Lemma \ref{lemmaAtilde}, $\tilde A$ is diagonalizable with $P$ such
that $P^{-1}\tilde AP=\Delta$ where $\Delta$ is a diagonal matrix.
Posit $X=P^{-1}\tilde\one$. One has
\[
P^{-1}\tilde A^kP=\Delta^k \; \Longrightarrow \;  
P^{-1}\tilde A^k\tilde\one=\Delta^k X \ , \quad k=1,\cdots
\]
This implies
\[
P^{-1}\left[\tilde\one,\tilde A\tilde\one,\cdots,\tilde A^{n-1}\tilde
  \one\right]=\diag(X)Vand(\lambda_{1},\cdots,\lambda_{n})
\]
or equivalently
\[
P^{-1}{\cal C}_{\tilde A,\tilde\one}=
\diag(X)Vand(\lambda_{1},\cdots,\lambda_{n-1})
\]
We deduce that when ${\cal C}_{\tilde A,\tilde\one}$ is full rank,
$\diag(X)$ and $Vand(\lambda_{1},\cdots,\lambda_{n-1})$ are
non-singular, that is all the entries of $X$ are non-null and the
eigenvalues $\lambda_{1},\cdots,\lambda_{n-1}$ are distinct.
We show now that the controllability of the pair $(A,B)$ implies that
the pair $(\tilde A,\tilde\one)$ is also controllable.\\

From the property $A\one=-B$, one can write
\[
A=\left[\begin{array}{rl}
A_{11} & L\\
-\tilde A\tilde \one & \tilde A
\end{array}\right]
\]
where $L$ is a row vector of length $n-1$. Then one has
\[
A\one=\left[\begin{array}{c}-1\\\tilde 0\end{array}\right] \ ,
\quad
A^2\one =\left[\begin{array}{c}-A_{11}\\\tilde A\tilde\one
    \end{array}\right] \ ,
\quad
A^3\one =\left[\begin{array}{c}-A_{11}^2+L\tilde A\tilde\one\\
A_{11}\tilde  A\tilde\one+\tilde A^2\one
    \end{array}\right]
\]
that are of the form
\[
A^k\one =\left[\begin{array}{c} \alpha_{k}\\
P_{k}
\end{array}\right] \mbox{ with }
P_{k}=\tilde A^{k-1}\one+\ds\sum_{j\leq k-2} \beta_{kj}\tilde A^j\one
\ ,
\mbox{ for } k=2,3
\]
By recursion, one obtains
\[
P_{k+1}=-\alpha_{k}\tilde A\tilde\one+\tilde A^{k}\one+\ds\sum_{j\leq k-2}
\beta_{kj}\tilde A^{j+1}\one
=\tilde A^{k}\one+\ds\sum_{j\leq k-1} \beta_{k+1,j}\tilde A^j\one
\ , \mbox{ for } k=2,\cdots
\]
Then, one can write
\[
-{\cal C}_{A,B}={\cal C}_{A,A\one}=\left[\begin{array}{llll}
\alpha_{1} & \alpha_{2} & \cdots & \alpha_{n}\\
\tilde 0  & P_{2} & \cdots & P_{n}
\end{array}\right]
\]
from which one deduces
\[
\rk ({\cal C}_{A,B})=n \; \Rightarrow \;
\rk(P_{2},\cdots,P_{n})=n-1 \; \Rightarrow \;
\rk(\tilde
A\one,\cdots\tilde A^{n-1}\one)=n-1
\]
One can also write
$[\tilde A\one,\cdots\tilde A^{n-1}\one]=\tilde A {\cal C}_{\tilde
  A,\tilde\one}$ and as $\tilde A$ is invertible (Lemma
\ref{lemmaAtilde}), we finally obtain that ${\cal C}_{\tilde A,\tilde\one}$
is full rank.
\end{proof}

Finally, Propositions \ref{mainprop} and \ref{propminimal} lead to the
main result of this section.

\begin{theorem}
\label{ThMRMT}
Any minimal representation $(A,B,C)$ that fulfills Assumptions \ref{H1} is
equivalent to a MRMT structure.
\end{theorem}

\section{Equivalence with MINC structure}
\label{secMINC}

Take a matrix $A$ that fulfills Assumption \ref{H1} and such that that
pair $(A,B)$ is controllable. As we have already shown that such representation $(A,B,C)$ is minimal and equivalent to a MRMT configuration, we can
assume without any loss of generality that the matrix $A$ has the
structure
\[
A=\left[\begin{array}{cc}
A_{11} & A(1,2:n)\\[2mm]
A(2:n,1) & \Delta
\end{array}\right]
\]
where $\Delta$ is a square diagonal matrix (of size $n-1$) with
distinct negative eigenvalues. We denote by $V$ the diagonal matrix of
the volumes associated to the matrix $A$ with $V_{1}=1$, as given by Lemma
\ref{lemmaV}
We shall consider a tridiagonalization of this
matrix. For this purpose, we recall the Lanczos algorithm.

\begin{defi} (Lanczos algorithm)
Let $S$ be a symmetric matrix of size $m$ and $q_{1}$ be a vector of
norm equal to one. One
defines the sequence $\pi_{k}=(\beta_{k},q_{k},r_{k})$ as follows
\begin{itemize}
\item $\beta_{0}=0$, $q_{0}=0$, $r_{0}=q_{1}$,
\item if $\beta_{k} \neq 0$, define $q_{k+1}=r_{k}/\beta_{k}$,
  $\alpha_{k+1}=q_{k+1}'Sq_{k+1}$,
  $r_{k+1}=(S-\alpha_{k+1}I)q_{k+1}-\beta_{k}q_{k}$ and $\beta_{k+1}=||r_{k+1}||$.
\end{itemize}
\end{defi}

One can straightforwardly check that the vectors $q_{k}$ provided by
this algorithm are orthogonal and of norm equal to one.
The algorithm stops for $k<m$ (``breakdown'') or $k=m$.
A {\em non-breakdown condition} for this algorithm is given in \cite[Th 10.1.1]{GV13}:

\begin{prop}
\label{Lanczos_conv}
When $\rk({\cal C}_{S,q_{1}})=m$, the sequence
$\pi_{k}$ is defined up to $k=m$, and the matrix $Q=[q_{1}\cdots
q_{m}]$ verifies
\[
Q'AQ=\left[\begin{array}{ccccc}
\alpha_{1} & \beta_{1} & & & \makebox(-10,-10){\text{\em \Large 0}}\\
\beta_{1} & \ddots & \ddots & & \\
  & \ddots & \ddots & \ddots & \\
  & & \ddots & \ddots & \beta_{m-1}\\
\makebox(20,20){\text{\em \Large 0}} & & & \beta_{m-1} & \alpha_{m}
\end{array}\right]
\]
where the numbers $\beta_{i}$ $(i=1\cdots m)$ are positive.
\end{prop}

\begin{lemma}
\label{lemLanczos}
The Lanczos algorithm applied to the matrix $\Delta$ with
$q_{1}=A(2:n,1)/||A(2:n,1)||$ provides an orthogonal unitary matrix
$Q$ such that $Q'\Delta Q$ is symmetric tridiagonal with positive terms on the
sub- (or super-) diagonal.
\end{lemma}

\begin{proof}
As the matrix $\Delta$ is diagonal, one has
\[
{\cal
  C}_{\Delta,q_{1}}=\vand(\lambda_{1},\cdots,\lambda_{n-1})\diag(q_{1})
\]
where $\lambda_{i}$ $(i=1\cdots n-1)$ are the diagonal elements of
$\Delta$.
Furthermore, as Assumptions \ref{H1} imply the equality
$A\one=-B$, one has
\[
q_{1}=-\frac{1}{\sqrt{\sum_{i=1}^{n-1}\lambda_{i}}}\left[\begin{array}{c}
\lambda_{1}\\
\vdots\\
\lambda_{n-1}
\end{array}\right]
\]
As $\lambda_{i}$ are all distinct and non null, $q_{1}$ is a non null
vector and the matrices
$\vand(\lambda_{1},\cdots,\lambda_{n-1})$, $\diag(q_{1})$ are full
rank. Therefore ${\cal C}_{\Delta,q_{1}}$ is full rank and Proposition
\ref{Lanczos_conv} can be used.
\end{proof}

Let us recall the well known Cholesky decomposition of symmetric
matrix.

\begin{theorem}
\label{Cholesky}
Let $S$ be a symmetric definite positive matrix. Then, there exists an
unique upper triangular matrix $U$ with positive diagonal entries such
that $S=U'U$.
\end{theorem}

We are ready now to explicit a tridiagonalization of the matrix $A$
with positive entries on the sub- and super-diagonals.

\begin{prop}
\label{prop-tridiag1}
Let $A$ be a MRMT matrix such that $(A,B)$ is controllable. 
Let $Q$ be the orthogonal matrix given by the Lanczos algorithm
applied to $\Delta $ with $q_{1}=A(2:n,1)/||A(2:n,1)||$. Let $U$ be
the upper triangular matrix with positive diagonal entries
given by the Cholesky decomposition of the symmetric matrix $Q'\tilde
V Q$.
Then the matrix
\[
T=\left[\begin{array}{cc}
1 & 0\\
0 & QU^{-1}
\end{array}\right] 
\]
is such that $T^{-1}AT$ is symmetric tridiagonal with positive entries on
the sub- (or super-)diagonal.
\end{prop}

\begin{proof}
Lemma \ref{lemLanczos} provides the existence of the matrix $Q$ such that
$Q'\Delta Q$ is tridiagonal with positive terms on the sub- and super-diagonal.
For convenience, we define the matrices
\[
P=\left[\begin{array}{cc}
1 & 0\\
0 & Q
\end{array}\right] \quad \mbox{and} \quad
W=\left[\begin{array}{cc}
1 & 0\\
0 & U
\end{array}\right] .
\]
Clearly, $P$ is orthogonal, $W$ is upper triangular with positive
diagonal, and one has $T=PW^{-1}$. Consider the matrix
\[
P'AP=\left[\begin{array}{cc}
A_{11} & A(1,2:n)Q\\[2mm]
Q'A(2:n,1) & Q'\Delta Q
\end{array}\right] .
\]
For the particular choice of the first column of $Q$, one has
\[
Q'A(2:n,1)=\frac{1}{||A(2:n,1)||}\left[\begin{array}{c}
1\\0\\\vdots\\0
\end{array}\right]
\]
and $Q'\Delta Q$ is triangular with positive sub-diagonal. Therefore,
$P'AP$ is an upper Hessenberg matrix with positive entries on its
sub-diagonal. Consider then 
\[
\begin{array}{lll}
P'{\cal C}_{A,B} & = & P'\left[\begin{array}{cccc}B & AB & A^2B & \cdots 
\end{array}\right]\\[2mm]
& = & \left[\begin{array}{cccc}P'B & (P'AP)P'B & (P'A^2P)P'B & \cdots 
\end{array}\right]
\end{array}
\]
Notice that one has $P'B=B$ and obtains recursively
\[
P'B=\left[\begin{array}{c} h_{1}\\ 0 \\ \vdots\\ \vdots \\0 
\end{array}\right], \quad
(P'AP)B=\left[\begin{array}{c} \star\\ h_{2}\\ 0\\ \vdots \\ \vdots \\0 
\end{array}\right], \quad
(P'A^2P)B=\left[\begin{array}{c} \star\\ \star\\ h_{3}\\ 0\\ \vdots \\0 
\end{array}\right], \quad \cdots
\]
where the number $h_{i}$ are positive.
Therefore, the matrix $P'{\cal C}_{A,B}$ is upper triangular with positive
diagonal, as the matrix $W$. Then $T^{-1}{\cal C}_{A,B}=WP'{\cal C}_{A,B}$ is also
upper triangular with positive entries on its diagonal.
Proposition \ref{PropGolub} implies that $T^{-1}AT$ is tridiagonal with
  positive entries on its sub-diagonal. Let us show that $T^{-1}AT$ is also
  symmetric. One has
\[
T^{-1}AT=\left[\begin{array}{cc}
A_{11} & A(1,2:n)QU^{-1}\\[2mm]
UQ'A(2:n,1) & UQ'\Delta QU^{-1}
\end{array}\right]
\]
As the matrix $VA$ is symmetric by Assumption \ref{H1}, one can write
\[
\begin{array}{lll}
\left(A(1,2:n)QU^{-1}\right)' & = & \frac{1}{V_{1}}(U^{-1})'Q'\tilde
VA(2:n,1)\\
& = & \frac{1}{V_{1}} (U^{-1})'U'UQ'A(2:n,1)\\
& = & \frac{1}{V_{1}}UQ'A(2:n,1)
\end{array}
\]
and as we have chosen $V_{1}=1$ we obtain
$\left(A(1,2:n)QU^{-1}\right)'=UQ'A(2:n,1)$. Consider now the
sub-matrix $UQ'\Delta QU^{-1}$.
Notice first that the decomposition
$Q'\tilde VQ=U'U$ implies the equalities $U'=Q'\tilde VQU^{-1}$ and
$(U^{-1})'=UQ'\tilde V^{-1}Q$. Then on can write
\[
\begin{array}{ccc}
\left(UQ'\Delta QU^{-1}\right)' & = & (U^{-1})'Q'\Delta QU'\\
& = & (UQ'\tilde V^{-1}Q)Q'\Delta Q(Q'\tilde VQU^{-1})\\
& = & UQ'\tilde V^{-1}\Delta\tilde VQU^{-1}\\
& = & UQ'\Delta QU^{-1}
\end{array}
\]
\end{proof}

The matrix $T$ provided by Proposition \ref{prop-tridiag1} possesses the following property.

\begin{prop}
\label{propX}
The vector $X=T^{-1}\one$, where the matrix $T$ is provided by
Proposition \ref{prop-tridiag1}, is positive.
\end{prop}

\begin{proof}
The matrices $A+BB'$ and $T^{-1}AT+BB'$ have non-negative entries
outside their main diagonals. So there exists a number $\gamma>0$ such
that $I+\frac{1}{\gamma}(A+BB')$ and $I+\frac{1}{\gamma}(T^{-1}AT+BB')$
are non-negative matrices.

By Assumption \ref{H1}, one has $A\one=-B$, which implies
 the property
\[
\left(I+\frac{1}{\gamma}(A+BB')\right)\one=\one \ .
\]
Thus $I+\frac{1}{\gamma}(A+BB')$ is a stochastic matrix, and we know
  that its maximal eigenvalue is $1$ (see for instance \cite[Th
  5.3]{BP94}).
As $I+\frac{1}{\gamma}(A+BB')$ and $I+\frac{1}{\gamma}(T^{-1}AT+BB')$ are
similar:
\[
T^{-1}\left(I+\frac{1}{\gamma}(T^{-1}AT+BB')\right)T=I+\frac{1}{\gamma}(A+BB') ,
\]
the maximal eigenvalue of $I+\frac{1}{\gamma}(T^{-1}AT+BB')$ is also $1$. Furthermore, as $A+BB'$ is irreducible by
Assumption \ref{H1}, $I+\frac{1}{\gamma}(T^{-1}AT+BB')$ is also
irreducible. The property $A\one=-B$ implies
\[
\left(I+\frac{1}{\gamma}(T^{-1}AT+BB')\right)X
=X+\frac{1}{\gamma}(T^{-1}A\one+B)=X+\frac{1}{\gamma}(-T^{-1}B+B)=X .
\]
So $X$ is an eigenvector of $I+\frac{1}{\gamma}(T^{-1}AT+BB')$ for its maximal
eigenvalue $1$. Finally, notice that $X=T^{-1}\one$ implies that the first entry of $X$
is equal to $1$. Then, by Perron-Frobenius Theorem (for non-negative
irreducible matrices, see for instance \cite[Th 1.4]{BP94})), we conclude that $X$ is a positive vector.
\end{proof}

We give now our main result concerning the MINC equivalence.

\begin{prop}
\label{propMINCequiv}
Let $A$ be a MRMT matrix such that $(A,B)$ is controllable and
$R=T\diag(T^{-1}\one)$, where $T$ is provided by Proposition
\ref{prop-tridiag1}. Then $(R^{-1}AR,B,C)$ is an equivalent
representation where $R^{-1}AR$ is a MINC matrix.
\end{prop}

\begin{proof}
Let $X=T^{-1}\one$ and $\bar A=R^{-1}AR$. Define $\bar V=\diag(X)^2$
and $\bar M=-\bar V(\bar A+BB')$. As $A+BB'$ is irreducible by
Assumption \ref{H1}, the similar matrix $\bar A+BB'$ is also
irreducible, as well as $\bar M$ because $V$ is a diagonal invertible matrix.

By Proposition \ref{prop-tridiag1}, $T^{-1}AT$ is a symmetric tridiagonal matrix with
positive terms on the sub- or super-diagonal. By Proposition \ref{propX},
$X$ is a positive vector, and thus $\bar
A=\diag(X)^{-1}(T^{-1}AT)\diag(X)$ is also a
tridiagonal matrix with the same signs outside the diagonal. Thus,
$\bar M$ is a tridiagonal matrix with negative terms on sub- or
super-diagonal. Moreover, one has
\[
\bar M=-\diag(X)^2(\diag(X)^{-1}T^{-1}AT\diag(X)+BB')
=-\diag(X)T^{-1}AT\diag(X)-X_{1}^2.BB' .
\]
where $X_{1}=1$.
The matrix $\bar M$ is thus symmetric. One has
\[
\begin{array}{lll}
\bar M\one & = & -\diag(X)T^{-1}AT\diag(X)\one-BB'\one\\
& = & -\diag(X)T^{-1}ATX-B\\
& = &  -\diag(X)T^{-1}AT(T^{-1}\one)-B\\
& = &  -\diag(X)T^{-1}A\one-B\\
& = &  \diag(X)T^{-1}B-B\\
& = & \diag(X)B-B\\
& = & 0
\end{array}
\]
The matrix $\bar M$ thus fulfills Assumption \ref{H1} and is
tridiagonal: $\bar A$ is then a MINC matrix. Finally, one has $\bar B=R^{-1}B=B$
and $\bar C=CR=C$.
\end{proof}

Finally, Theorem \ref{ThMRMT} and Proposition \ref{propMINCequiv} lead to the
following result.

\begin{theorem}
\label{ThMINC}
Any minimal representation $(A,B,C)$ that fulfills Assumptions \ref{H1} is
equivalent to a MINC structure.
\end{theorem}

\section{Examples and discussion}
\label{section-minimal}

Theorems \ref{ThMRMT} and \ref{ThMINC} show that whatever is the network
structure, it is always possible to represent its input-output map
with either a MRMT {\em star} or a
MINC {\em series} structure.
But from a state-space representation (\ref{input-output}), it
requires the system to be controllable (or minimal).
We begin by an example that illustrates the
necessity of  the controllability assumption to make the equivalence
constructions given in Sections \ref{secMRMT}, \ref{secMINC} work.

\subsection{Example 1}
Consider a network of four reservoirs of volumes (see Fig. \ref{fig-exemple})
\[
V_{1}=1 \ , \; V_{2}=1 \ , \; V_{3}=2 \ , \; V_{4}=3
\]
with the diffusive exchange rate coefficients:
\[
d_{12}=1 \ , \; d_{13}=2 \ , \; d_{14}=3 \ , \; d_{23}=3 \ , \;
d_{24}=3
\]
which lead to the dynamics
\[
\begin{array}{lll}
\dot S_{1} & = & -7S_{1}+S_{2}+2S_{3}+3S_{4}+u\\
\dot S_{2} & = & S_{1}-7S_{2}+3S_{3}+3S_{4}\\
\dot S_{3} & = & S_{1}+\frac{3}{2}S_{2}-\frac{5}{2}S_{3}\\
\dot S_{4} & = & S_{1}+S_{2}-2S_{4}
\end{array}
\]
with the matrix
\[
A=\left[\begin{array}{rrrr}
-7 & 1 & 2 & 3\\
1 & -7 & 3 & 3\\
1 & \frac{3}{2} & -\frac{5}{2} & 0\\
1 & 1 & 0 & -2
\end{array}\right]
\]
\begin{figure}[h]
\begin{center}
\includegraphics[width=8cm]{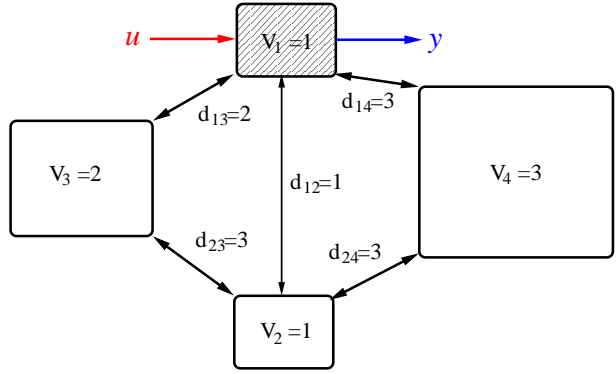}
\caption{\label{fig-exemple}Structure of the example}
\end{center}
\end{figure}
At the first look, this structure does not exhibit any special
property or symmetry that could make believe that it is non minimal.
By construction one has $\tilde A\tilde\one=-\tilde\one$ but the particular
matrix $A$ that we consider satisfies
$A(2:4,1)=\tilde\one$. Consequently the vector $A(2:4,1)$ is an
eigenvector of the matrix $\tilde A$ for the eigenvalue $-1$.
If the multiplicity of $-1$ was more than $1$, then $\lambda=-9.5$ should be an
eigenvalue of $\tilde A$, as the trace of $\tilde A$ is $-11.5$.
But an eigenvector $X$ of $\tilde A$ fulfills
\[
\begin{array}{rll}
-7X_{1}+3X_{2}+3X_{3} & = & \lambda X_{1}\\
1.5X_{1}-2.5X_{2} & = & \lambda X_{2}\\
X_{1}-2X_{3} & = & \lambda X_{3}
\end{array}
\]
one should have
\[
(\lambda+7)X_{1}-3X_{2}-3X_{3}=0 \mbox{ with }
X_{2}=\frac{1.5}{\lambda+2.5}X_{1} \ , \;
X_{3}=\frac{1}{\lambda+2}X_{1} \quad (\mbox{and } X_{1}\neq 0)
\]
which is not possible for $\lambda=-9.5$. Then, any matrix $P$ that
diagonalizes $\tilde A$ should have one column proportional to the
eigenvector $\tilde \one$, which amounts to have the vector
$P^{-1}\tilde\one$ with exactly one non-null entry. 
Thus one cannot apply Proposition \ref{mainprop} and transform the system in a equivalent MRMT structure of the same dimension.\\

One can check that the pair $(A,B)$ is indeed non controllable, even though
the matrix $\tilde A$ has distinct eigenvalues, as one has
\[
AB=\left[\begin{array}{r}-7\\1\\1\\1\end{array}\right] \ , \;
A^2B=\left[\begin{array}{r}55\\-8\\-8\\-8\end{array}\right] =-B-8AB
\]
from which one deduce $\rk({\cal C}_{A,B})=2$. Indeed, the system admits a
minimal representation of dimension $2$ that can be found by gathering
the immobile zones in one of volume $\bar V=V_{2}+V_{3}+V_{4}=6$ and
solute concentration
\[
\bar
S=\frac{V_{2}S_{2}+V_{3}S_{3}+V_{4}S_{4}}{\bar V}=\frac{S_{2}+2S_{3}+3S_{4}}{6}
\]
One can check that variables $(S_{1},\bar S)$ are solutions
of the dynamics
\[
\begin{array}{lll}
\dot S_{1} & = & -7S_{1}+6\bar S+u\\
\dot{\bar S} & = & S_{1}-\bar S
\end{array}
\]
that gives an equivalent representation (in MRMT or MINC form) with a
diffusive exchange rate $\bar d=6$ (see Fig. \ref{fig-exemple-simplified}).
\begin{figure}[h]
\begin{center}
\includegraphics[width=4cm]{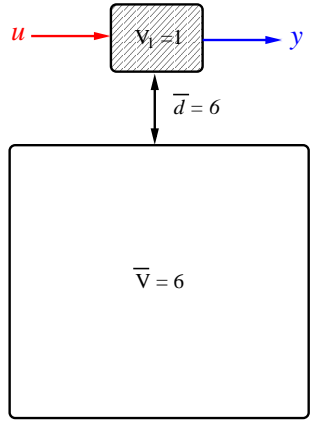}
\caption{\label{fig-exemple-simplified}Simplified equivalent structure of the example}
\end{center}
\end{figure}

\subsection{Example 2}
Consider a network with one mobile zone and four immobile zones of
identical volumes $V_{i}=1$ ($i=1\cdots 5$), as depicted on Figure
\ref{figEx5} with the following diffusive exchange rates
\[
d_{12}=1 \ , \; d_{13}=2 \ , \; d_{34}=1 \ , \; d_{35}=3 \ , \;
d_{45}=1 \ .
\]
\begin{figure}[h]
\begin{center}
\includegraphics[width=6cm]{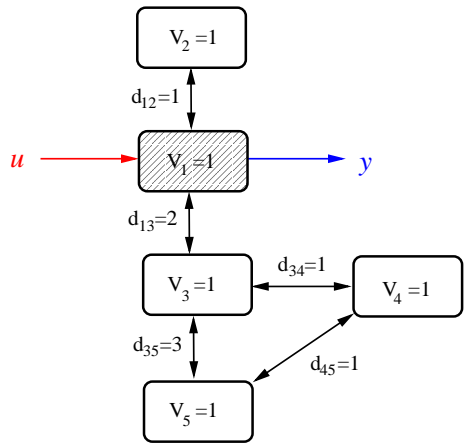}
\caption{\label{figEx5}Example of a network with one mobile and four immobile zones}
\end{center}
\end{figure}

The structure of this network is neither MRMT nor MINC, and its
corresponding matrix $A$ is
\[
A=\left[\begin{array}{ccccc}
   -3 &     1 &    1 &     0 &     0 \\
    1 &   - 1 &     0 &     0 &     0\\  
    1 &     0 &   -3 &     1 &     1\\
    0 &     0 &     1 &   -2 &     1\\  
    0 &    0 &    1 &     1 &   -2
\end{array}\right] .
\]
One can easily compute the controllability matrix
\[
{\cal C}_{A,B}=\left[\begin{array}{ccccc}
    1 & -4 &  21 & -129 &  906 \\ 
    0 &  1 & -5 &   26 & -155 \\
    0 &  2 & -20 &  182 & -1614 \\
    0 &  0 &  2 & -18 &   136 \\
    0 &  0 &  6 & -82 &   856
\end{array}\right]
\]
and check that it is full rank (computing for instance $det({\cal C}_{A,B})=-896$). Then, the
constructions of Sections \ref{secMRMT} and \ref{secMINC} give the
following equivalent MRMT and MINC matrices:
\[
A_{MRMT}=\left[\begin{array}{ccccc}
 -4 & 0.3256267 & 0.1692779  & 1 &  1.5050954 \\
    8.1710298  & -8.1710298  &  0 &  0 &  0 \\
    3.3115831 &   0 &       -3.3115831 &   0 &  0\\
    1 &         0 &         0 &       - 1 &  0 \\    
    0.5173871  &  0 &         0 &         0 & -0.5173871  
\end{array}\right] ,
\]
\[
A_{MINC}=\left[\begin{array}{ccccc}
 -4 &          3 &          0 &          0 &          0 \\       
    1.6666667  & -5 &          3.3333333  &  0 &       0 \\ 
    0 &          3.6   &     -4.1333333 &   0.5333333  & 0\\ 
    0 &          0 &    2.4666667  & -2.9207207 &   0.4540541 \\  
    0 &          0 &  0 &    0.9459459 & -0.9459459
\end{array}\right] .
\]
We have checked numerically that each matrix $A$, $A_{MRMT}$ and
$A_{MINC}$ give the same co-prime transfer function
\[
T(z)=\frac{ 14 + 47z + 45z^2 + 13z^3 + z^4}{14 + 117z + 187z^2 + 92z^3
  + 17z^4 + z^5} .
\]

Differently to the original network, the magnitude of the
values of volumes and diffusive exchange rates are significantly
different among compartments, opening the door of possible model
reduction dropping some compartments.
\begin{enumerate} 
\item[i.] For the equivalent MRMT structure, one obtains
\[
V_{1}=1 \ , \; V_{2}=0.0398514 \ , \; V_{3}=0.0511169 \ , \; V_{4}=1 \
, \; V_{5}=2.9090317 
\]
with
\[
d_{12} =0.3256267  \ , \; d_{13}=0.1692779 \ , \; d_{14}=1 , \;
d_{15}=1.5050954
\]
and notices that zones $2$ and $3$ are of relatively small volumes
(compared to the total volume of the system which is equal to $5$) and connected to the mobile zone with relatively small diffusive
parameters. Then, one may expect to have a good approximation with a reduced
MRMT model dropping zones $2$ and $3$. Keeping the volumes $V_{1}$,
$V_{4}$, $V_{5}$ with the parameters $d_{14}$, $d_{15}$, one obtains the
the 3 compartments MRMT matrix
\[
\tilde A_{MRMT}= \left[\begin{array}{ccc}
  -3.5050954  &  1 &           1.5050954 \\ 
    1 &         - 1 &           0 \\
    0 &        0.5173871  & -0.5173871
\end{array}\right]
\]
with the corresponding transfer function
\[
\tilde T_{MRMT}(z)=\frac{0.5173871 + 1.5173871z + z^2}{0.5173871 +
  4.8359736z + 5.0224825z^2 + z^3} \ .
\]
\item[ii.]
For the equivalent MINC structure, one obtains
\[
V_{1}=1 \ , \; V_{2}=1.8 \ , \; V_{3}=1.6666667 \ , \; V_{4}=0.3603604 \
, \; V_{5}=0.1729730 
\]
with
\[
d_{12} =3  \ , \; d_{23}=6 \ , \; d_{34}=0.8888889 , \;
d_{45}=0.1636231 \ .
\]
Here, one notices that the two last volumes are relatively small and
connected with relatively small diffusion terms. 
Keeping the volumes $V_{1}$,
$V_{2}$, $V_{3}$ with the parameters $d_{12}$, $d_{23}$, one obtains the
the 3 compartments MINC matrix
\[
\tilde A_{MINC}= \left[\begin{array}{ccc}
  -4 &  3 &           0 \\ 
1.6666667  & -5  &     3.3333333  \\
    0        & 3.6 & -3.6 
\end{array}\right]
\]
with the corresponding transfer function
\[
\tilde T_{MINC}(z)=\frac{6 + 8.6z + z^2}{6 + 35.4z + 12.6z^2 + z^3} \ .
\]
\end{enumerate}
The Nyquist plots of the transfer functions $T$, $\tilde T_{MRMT}$ and
$\tilde T_{MINC}$ are reported on Figure \ref{figNyquist}, showing the quality
of the approximation with only three compartments derived from the MRMT or MINC
representations. There exist many reduction methods in the literature,
but a reduction through MRMT or MINC has the advantage to obtain easily reduced
models with a physical meaning.
\begin{figure}[h]
\begin{center}
\includegraphics[width=6cm]{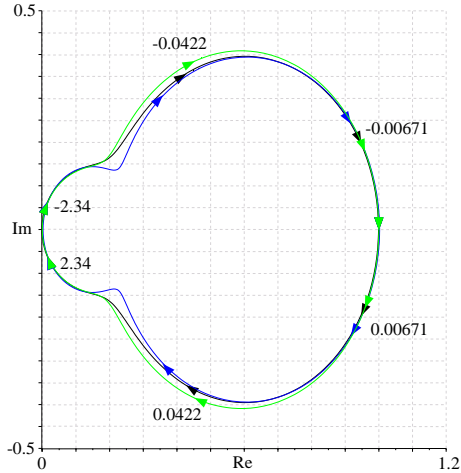}
\caption{\label{figNyquist}Nyquist diagrams (black: original system,
  blue: reduced MRMT, green: reduced MINC)}
\end{center}
\end{figure}

\begin{remark}
For a positive linear system $(A,B)$, let ${\cal A}_{0}^+(A,B)$ be the attainability set from the $0$-state with
non-negative controls. The system being positive, one has
${\cal A}_{0}^+(A,B) \subset \Rset_{+}^n$ and for any state $X \in
{\cal A}_{0}^+(A,B)$, the state $Z=R^{-1}X$ for the equivalent
MRMT or MINC structure is also non-negative, but for a state $X \in \Rset_{+}^n
\setminus {\cal A}_{0}^+(A,B)$, the equivalent state $Z=R^{-1}X$ is not
necessarily non-negative (as the coefficients of the matrix $R^{-1}$
are not necessarily non-negative).
Consequently, one can have an equivalent input-output representation
in MRMT form but with negative concentrations for such states of the
system.
Physically, this means that the compartments network has not been
filled from a substrate-free state with a control $u(\cdot)$.This
observation might be relevant from a geophysical view point.
\end{remark} 

\section{Conclusion}

We have shown that any general network structure is equivalent to a
``star'' structure (MRMT) or a ``series'' structure (MINC), that are
commonly considered in geosciences to represent soil porosity in mass
transfers. In this way, we reconcile these two different approaches, showing
that they are indeed equivalent. Practically, this means that when the
structure is unknown, or partially known, one can use equivalently
the most convenient structure to identify the parameters or use some a
priory knowledge.

In this work we have also shown the crucial role played the
controllability property of a given mass transfer structure. Although there is no particular
control issue in the input-output representations of mass transfers,
controllability is a necessary condition to obtain equivalence with
the multi-rate mass transfer (MRMT) structures of depth one, introduced by
Haggerty and Gorelick in 1995 \cite{HG95}, or the multiple interacting
 continua structure (MINC).
This condition is related to the minimal representation of linear
systems, that is not necessarily fulfilled for such structures even
for non-singular irreducible network matrices with distinct eigenvalues.

Although the objective of the present work is to show the exact
equivalence of systems, we have shown on examples that MRMT and MINC
representations could allow a simple and efficient way to obtain reduced
models with a good approximation. Further investigations about such
reduction techniques will be the matter of a coming work.

From a geosciences view point, this analysis shows the existence of
both identifiable and non-identifiable porosity structures from
input-output data. Input-output signals are typical of conservative
tracer tests where non-reactive tracers are injected in an upstream
well and analyzed in a downstream well \cite{F08}. Identifiable
structures could thus be calibrated on tracer tests \cite{A83}. The
porosity structure identified is however not unique as demonstrated on
the example in Section \ref{section-minimal}, meaning that a porosity
structure cannot be fully characterized by a tracer test. This is an
advantage rather than a drawback for this class of models as the
porosity structure should support both conservative and reactive
transport \cite{DRBH13,DSSCB09}. Reactive transport does not only
depend on the input/output concentrations but also on the
concentrations within the diffusion porosities, i.e. from the full
state of the system. 
In a broader perspective, some further characteristics of the porosity structure might be revealed by reactive tracers used in combination with conservative tracers. 

\bigskip

\noindent {\bf Acknowledgments.}

This work was developed in the framework of the DYMECOS 2 INRIA Associated team and of project BIONATURE of CIRIC INRIA CHILE, and it was partially supported by CONICYT grants REDES 130067 and 150011. The second and fourth authors were also supported by CONICYT-Chile under ACT project 10336, FONDECYT 1160204 and 1160567, BASAL project (Centro de Modelamiento Matemático, Universidad de Chile), CONICYT national doctoral grant and CONICYT PAI/ Concurso Nacional Tesis de Doctorado en la Empresa, convocatoria 2014, 781413008. 

The authors are grateful to T. Babey, D. Dochain J. Harmand, C. Casenave,
J.L. Gouz\'e and B. Cloez for fruitful discussions and insightful
ideas.

This research has been also conducted in the scope of the French ANR
project {\em Soil$\mu$3D}.

\end{document}